\documentclass[a4paper, 12pt]{article}
\usepackage{amsfonts}
\usepackage{fancyhdr}
\usepackage{graphicx}
\usepackage{epsfig}
\usepackage{amssymb}
\usepackage{mathrsfs,amsfonts,amsmath}
\usepackage{color}
\usepackage{array}
\usepackage{tabularx,multirow,makecell}
\usepackage{ragged2e}
\usepackage{float}

\makeatletter

\@addtoreset{equation}{section} \makeatother

\newtheorem{Theorem}{Theorem}[section]
\newtheorem{Remark}[Theorem]{Remark}
\newtheorem{Lemma}[Theorem]{Lemma}

\newtheorem{Assumption}[Theorem]{Assumption}
\newtheorem{Condition}[Theorem]{Condition}

\begin{document}

\title{\textbf{Stabilization of highly nonlinear hybrid stochastic differential delay equations by periodically intermittent feedback controls based on discrete time observations with asynchronous switching}}
\author{Guangqiang Lan,\footnote{Corresponding author: Email: langq@buct.edu.cn.}\quad Fansai Meng \quad and\quad Ge Wang
\\ \footnotesize College of Mathematics and Physics, Beijing University of Chemical Technology, Beijing 100029, China}

\date{}
\maketitle

\begin{abstract}
This paper extends the stabilization analysis of highly nonlinear time delay systems to the case where periodic intermittent control, discrete time observation, and sampling-induced asynchronous mode switching coexist, and establishes corresponding moment estimates and exponential stability criteria. To overcome the difficulties emerged due to intermittent control and discrete time observations with asynchronous switching, we introduce a discontinuous Lyapunov function and account for the effect of the discontinuities of the intermittent control function at switching instants, thereby establishing the moment exponential stability of the controlled system. Finally, an example is presented to interpret the conclusion.
\end{abstract}

\noindent\textbf{MSC 2010:} 60H10.

\noindent\textbf{Key words:} hybrid stochastic differential delay equations; highly nonlinear; discrete-time observations; periodically intermittent feedback control; exponential stability.

\section{Introduction}

In reality, hybrid stochastic systems (or stochastic systems with switching) can be used to describe many interested models since the systems may experience abrupt changes in their structure and parameters due to component failures, changing subsystem interconnections and environmental disturbances, etc. Since the stability is a very important property of a given system, it is very natural that stabilization of given unstable systems by feedback controller becomes to one of the important issues of hybrid stochastic systems.

There are already many results for stabilization of hybrid stochastic systems with continuous time or discrete time control. For example, \cite{CXZ} stabilizes neutral
stochastic differential equations by delay feedback control, but they only consider the case when $f$ and $g$ are both linear growing functions, \cite{MFSFM,SFFM} generalizes it to highly nonlinear case, \cite{XM} gives advances in discrete-state-feedback stabilization of highly nonlinear hybrid stochastic functional differential equations by Razumikhin technique, \cite{DM}, \cite{SF} and \cite{WYZ} consider stabilization of highly nonlinear stochastic differential delay equations by variable-delay feedback control, while \cite{FLCA} considers constant delayed control with asynchronous switching by using different methods, \cite{DTM} stabilizes highly nonlinear stochastic differential delay equations with L\'evy noise by time-varying delay feedback control. For the asynchronous switching cases, \cite{shao} investigates stabilization of regime-switching processes by feedback control based on discrete time observations when coefficients of the underline systems are both Lipschitz continuous.

Since the intermittent control has great efficiency and much lower cost compared with the continuous control technique, then stabilization of hybrid stochastic systems with intermittent control has attracted more and more attention. For instance, \cite{shen} and \cite{HRZ} consider stabilization of highly nonlinear hybrid stochastic differential equations without delays by aperiodic and periodic intermittent feedback control, respectively. \cite{LPC} considers aperiodically intermittent stochastic stabilization via discrete time or delay feedback control of given deterministic system when the coefficients are all Lipschitz continuous. When there is delay term, \cite{WLS} investigate stabilization for complex-valued hybrid stochastic delayed systems with aperiodically intermittent control by using average technique, however, the diffusion term there needs to be linearly growing. \cite{JHLMM} considers periodically intermittent control for hybrid systems based on discrete observations under Lipschitz conditions, while \cite{WDLY} considers the problem of aperiodically intermittent control for neutral stochastic delay systems, auxiliary controlled hybrid systems are introduced to obtain the stability of the intermittent feedback controlled system, and \cite{MJHM} investigates stabilization of hybrid ordinary differential equations with intermittent stochastic noise, see also \cite{MYJM} for intermittent stochastic stabilization. Moreover, all the coefficients (drift, diffusion and control functions) in \cite{JHLMM,WDLY,MJHM,MYJM} need to be Lipschitz continuous functions.

However, to the best of our knowledge, few results focus on the stabilization problem for stochastic differential delay equations under intermittent control based on discrete-time observations with asynchronous switching. In this paper, we therefore extend the stabilization analysis of highly nonlinear time delay systems to the case where periodic intermittent control, discrete time observation, and sampling-induced asynchronous mode switching coexist, and establish corresponding moment estimates and moment exponential stability.

The main contributions of this paper are threefold. First, we establish moment exponential stabilization criteria for highly nonlinear hybrid stochastic differential delay equations under periodically intermittent feedback based on discrete time observations of both the system state and the switching mode. The results accommodate polynomial growth in both the drift and diffusion coefficients, together with system delays and sampling-induced asynchronous switching. Second, we derive moment boundedness in which estimation of the delayed control function term must be retained, and obtain the moment exponential stability of the controlled system by developing a Lyapunov functional (which is discontinuous at the ends of intervals $(nT,nT+\theta)$) addressing sampling errors and discontinuities at control activation and deactivation times. Third, we obtain quantitative sufficient bounds on the observation interval and control width, together with an explicit estimate of the exponential decay rate.

The rest of the paper is organized as follows. In Section 2, we will introduce the the underline system and some standing Assumptions. In Section 3, moment finiteness and boundness of the controlled systems are obtained under suitable conditions on the coefficients. Then we will carefully design the control function in order to stabilize the underline SDDEs and present the main result and prove it in Section 4. An example is provided to interpret the conclusion in Section 5. Finally, we give a concluding remark in Section 6.

\section{Model description and Assumptions}

Throughout this paper, unless otherwise specified, we use the following notation. Let $S=\left\lbrace 1,2,\cdots,N\right\rbrace$ and $R_{+}=\left[ 0,\infty\right)$, where $N$ is a positive integer. For $x\in R_{n}$, $|x|$ denotes its Euclidean norm. If A is a vector or matrix, its transpose is denoted by $A^\textrm{T}$. By $A \le 0$ and $A<0$, we mean $A$ is non-positive and negative definite, respectively. If $A$ is a matrix, we let $\left|A\right| =\sqrt{\textrm{trace}(A^\textrm{T}A)}$ be its trace norm.
Let $(\Omega,\mathcal{F},\{\mathcal{F}_{t}\} _{t\geq 0},P)$ be a complete probability space with a filtration  $\{\mathcal{F}_{t}\} _{t\geq 0}$ satisfying the usual conditions (that is, it is increasing and right continuous while $\mathcal{F}_{0}$ contains all $P$-null sets). Let $r(t)$, $t\geq0$, be a right-continuous Markov chain on the probability space taking values in a finite state space $S=\{1, 2, \dots, N\}$ with transition probability given by
$$P(r(t+\Delta)=j|r(t)=i)=\left\{\begin{array}{ll}
		\gamma_{ij}\Delta+o(\Delta),\quad  i\neq j ,
		\\
		1+\gamma_{ij}\Delta+o(\Delta),\quad i= j
	\end{array}
	\right.$$
where $\Delta>0$.
Here $\gamma_{ij}\geq0$ is the transition rate from $i$ to $j$ if $i\neq j$ while $\gamma_{ii}=-\sum_{j\neq i}\gamma_{ij}$. Let $B(t)=(B_{1}(t),\dots, B_{m}(t))^{T}$ be an m-dimensional Brownian motion defined on the probability space. We assume that the Markov chain $r(\cdot)$ is independent of the Brownian motion $B(\cdot)$. We stress that almost all sample paths of $r(t) $ are right continuous. It is well known that almost all sample paths of $r\left( t\right) $ are piecewise constant except for a finite number of simple jumps in any finite subinterval of $R_{+}$. If $A$ is a subset of $\Omega$, denote by $I_{A}$ its indicator function. If both $a$, $b$ are real numbers, then $a\wedge b=\min\{ a,b\}$ and $a\vee b=\max\{a,b\}$. For $\tau>0$, denote by $C([-\tau, 0]; R^{n})$ the family of continuous functions $\varphi$ from $[-\tau,0]\to R^{n}$ with the norm $||\varphi||=\sup_{-\tau\le u\le0}|\varphi(u)|$.
	$E$  is the expectation with respect to the probability measure $P$. For a non-negative real number $a$, let $[a]$ denote the integer part of $a$.

Consider an unstable hybrid SDDE:
	\begin{equation}\label{sde}\begin{aligned} dx(t)
		&=f(x(t),x(t-h(t)),r(t),t)dt+g(x(t),x(t-h(t)),r(t),t)dB(t)\end{aligned}
	\end{equation}
on $t\geq 0$ with initial value $x_{0}=\xi(t)\in C([-\tau, 0]; R^{n})$ and $r(0)=r_{0}$, where
	$f: R^n\times R^n \times S \times R_{+} \to R^n$
	and $g: R^n \times R^n \times S \times R_{+} \to R^{n\times m}$ are Borel measurable functions, and $h(t)$ is a Borel measurable function from $R_+$ to $[h',\tau]$ for some $h'\ge0$. We aim to design the intermittent control function $u(x,i,t)I(t)$ and the gap of the discrete time state observation $\delta>0$ to make it stable. Thus, the controlled system is as follows:
	\begin{equation}\label{jsde}
\begin{aligned}
dx(t) &= \left[f(x(t),x(t-h(t)),r(t),t)+ u(x(v(t)),r(v(t)),t)I(t)\right]dt \\
      &\quad + g(x(t),x(t-h(t)),r(t),t)dB(t)
\end{aligned}
\end{equation}
where $v(t)=\left\lfloor t/\delta\right\rfloor\delta$ and
$I(t)=\sum_{n=0}^\infty I_{\left[nT,nT+\theta\right)}(t), \theta \in\left[0,T\right].$ That is, the controller only depends on the system and switching states $x(k\delta), r(k\delta)$ on interval $t\in [k\delta,(k+1)\delta).$

To make the model more general, throughout this paper, suppose the variable time delay function satisfies
\begin{Assumption}\label{delay}
\[{h^*}:= \mathop {\limsup }\limits_{\Delta  \to {0^+ }} \left( {\mathop {\sup }\limits_{s \ge  - \tau } \frac{{\mu \left( {{I_{s,\Delta }}} \right)}}{\Delta }} \right)<\infty \]
here $I_{s,\Delta} = \{ t \in \mathbb{R}_+ :
  \begin{aligned}[t]
    t - h(t) \in [s, s+\Delta)
  \end{aligned} \}$, $\mu ( \cdot )$ indicates the Lebesgue measure on the ${\mathbb{R}_+ }$.
\end{Assumption}
We need the following result introduced in \cite{DM}.
\begin{Lemma}\label{lemma}
Suppose that Assumption 1 holds. Then for  any continuous function $\varphi:\left[ { - \tau ,T - h'} \right] \to {\mathbb{R}_+ }, T>h'>0,$ it follows that
$$\int_0^T\varphi\left({t-h\left( t \right)} \right)dt \le {h^* }\int_{-\tau}^{T-h'} \varphi\left(t\right)dt$$
\end{Lemma}

Obviously, $h^*\ge1.$

Moreover, the coefficients are assumed to satisfy the following assumption.
	
	\begin{Assumption}\label{one}
		Assume that both the coefficients $f$ and $g$ in (\ref{sde}) are locally Lipschitz continuous, that is, for each $b>0$ there is $K_b>0$ (depending on $b$) such that
		\begin{equation}\label{local}
        \begin{aligned}
        &\quad |f(x,y,i,t)-f(\bar{x},\bar{y},i,t)| \vee |g(x,y,i,t)-g(\bar{x},\bar{y},i,t)|\le K_{b}(|x-\bar{x}|+|y-\bar{y}|)
        \end{aligned}
        \end{equation}
		for all	$|x|\vee|y|\vee|\bar{x}|\vee|\bar{y}|\le b$  and all $\left( i,t\right) \in S\times R_{+}$.
	\end{Assumption}

\begin{Assumption}\label{polynomial}
Assume also that there exist constants $K>0, q_{1}>1$ and $q_{i}\geq1, i=2,3,4$ such that
		\begin{equation}\label{duoxiangshi}
        \begin{aligned}
        |f(x,y,i,t)| &\le K(|x|+|y|+|x|^{q_{1}}+|y|^{q_2})\\
        \quad |g(x,y,i,t)| &\le K(|x|+|y|+|x|^{q_{3}}+|y|^{q_4})
        \end{aligned}
        \end{equation}
		for all $\left( x,y,i,t\right) \in R^{n}\times R^n\times S\times R_{+}$.
	\end{Assumption}
	
	Condition (\ref{duoxiangshi}) implies $f\left(0,0,i,t\right)\equiv 0$ and $g\left(0,0,i,t\right)\equiv 0$,
	which are required for the stability purpose of this paper. Of
	course, we also need a Khasminskii type condition, which guarantees the nonexplosion of the unique global solution of the
	hybrid SDDE (\ref{jsde}).
	\begin{Assumption}\label{Khas}
		Assume that there exist positive constants $p, q, q_i, i=1,\cdots, 4, \alpha_i, =1,2$ such that
		\begin{equation}\label{pq}
        \begin{aligned}
        q &> p+q_{1}-1\ge 2(q_{1}\vee q_2\vee q_3\vee q_4) \\
        &\quad \quad p \ge 2(q_{1}\vee q_2\vee q_3\vee q_4)-q_1+1
        \end{aligned}
        \end{equation}
		while
		\begin{equation}\label{khas}
        \begin{aligned}
        &\quad x^{T}f(x,y,i,t) + \dfrac{q-1}{2}|g(x,y,i,t)|^{2}\le K'(|x|^2 + |y|^2) - \alpha_1|x|^{p} + \alpha_2|y|^{p}
        \end{aligned}
        \end{equation}
		for all $\left( x,y,i,t\right) \in R^{n}\times R^n\times S\times R_{+}$.
	\end{Assumption}
	In this paper, the control function  satisfies the following assumption.
	\begin{Assumption}\label{u1}
		Assume that there exists a positive number $L$ such that
		\begin{equation}\label{uu}
			|u(x,i,t)-u(\bar{x},i,t)|\le L(|x-\bar{x}|).
		\end{equation}
		Moreover, for the stability purpose, assume that $u(0, i, t)\equiv 0$.
		for all $\left( x,i,t\right) \in R^{n}\times S\times R_{+}$
		which implies \begin{equation}\label{2.10}|u(x,i,t)|\le L|x|.\end{equation}	
	\end{Assumption}

\section{$p$-th moment finiteness and boundedness}

	\begin{Theorem}\label{youjie1}
		Under Assumptions \ref{delay}, \ref{one}, \ref{Khas} and \ref{u1}, the controlled system (\ref{jsde}) with any initial value $x_{0}=\xi$ has a unique global solution $x(t)$ on $t\geqslant 0$ and for all $t>0$
\begin{equation}\label{l6}
\sup_{0\le s\le t}\mathbb{E}\left|x(s)\right|^q < \infty
\end{equation}
and
\begin{equation}\label{l7}
\mathbb{E}\int_0^t\left|x(s)\right|^{p+q-2}ds < \infty.
\end{equation}
Furthermore, if Assumption \ref{Khas} holds for
$$
\alpha_1 - \alpha_2 \frac{q - 2 + p h^*}{p + q - 2} > 0,
$$
then the solution has the property that
\begin{equation}\label{youjie112}
\sup_{0\le t< \infty} \mathbb{E}|x(t)|^q < \infty.
\end{equation}
	\end{Theorem}
\begin{Remark}
Note that although Theorem 3.3 in \cite{DM} claims that the controlled system holds the same $q$-th moment boundedness as the uncontrolled system, the proof of Theorem 2.6 in \cite{DM} can not be applied directly here since the term $\int_{0}^{t}e^{\lambda s}\mathbb{E}|x(v(s))|^qds$ appears in the following (\ref{l10'}) on the right hand side due to discrete observation of the control term $u$. Indeed, in the proof of Theorem 2.6 in \cite{DM},
$$
e^{\lambda t} \mathbb{E}\left|x(t)\right|^q
\le C_1 +C_2\frac{e^{\lambda t}}{\lambda}
$$ are used to obtain the boundedness of $q$-th moment of uncontrolled system.
 So we need to use Lemma \ref{l1} to obtain the boundedness of $q$-th moment of $x(t)$.

\end{Remark}

To prove this Theorem, we need the following Lemma:
\begin{Lemma}\label{l1}
Suppose $C_1, C_2, C_3 $ and $\lambda$ are all positive constants. $x(t), t\ge0$ is a stochastic process. If
\begin{equation}\label{l10}
e^{\lambda t} \mathbb{E}\left|x(t)\right|^q\le C_1+C_2\frac{e^{\lambda t}}{\lambda}+C_3\delta\int_{0}^{t}e^{\lambda s}\mathbb{E}|x(v(s))|^qds
\end{equation}
for any $t\ge0,$ then for sufficiently small $\delta>0,$ we have
\begin{equation}\label{sup_eq}
\begin{aligned}
\sup_{t\ge0}\mathbb{E}\left|x(t)\right|^q
&\le C_1e^{-\lambda t} + \frac{C_2}{\lambda} + \left(C_1 +\frac{C_2}{\lambda}\frac{e^{\lambda\delta}-1}{e^{\lambda\delta}-1-\alpha}\right)\frac{C_3\delta e^{\lambda\delta} }{\lambda} \\
&< \infty
\end{aligned}
\end{equation}
where $\alpha = \frac{C_3 \delta (e^{\lambda \delta}-1)}{\lambda}$.
\end{Lemma}
\textbf{Proof.} For $t=k\delta,$ if we denote $a_k=\mathbb{E}\left|x(k\delta)\right|^q$, $b_n = e^{\lambda t_n} a(t_n)$ and $S_n = \sum_{k=0}^{n-1} b_k$, then we have
$$b_k \le C_1 + \frac{C_2}{\lambda} e^{k\lambda \delta} + \alpha S_k$$
 and thus
\begin{equation}\label{S_k_eq}
\begin{aligned}
S_{k+1} &=S_k + b_k \\
        &\le S_k + C_1 + \frac{C_2}{\lambda} e^{k\lambda \delta} + \alpha S_k \\
        &\le (1+\alpha) S_k + C_1 + \frac{C_2}{\lambda} e^{k\lambda \delta}
\end{aligned}
\end{equation}
with initial value $S_0=0$.
Consequently,
$$S_k\le\sum_{i=0}^{k-1}(C_1 +\frac{C_2}{\lambda} e^{i\lambda \delta})(1+\alpha)^{n-1-i}.$$
So
\begin{equation}\label{b_k_eq}
\begin{aligned}
b_k &\le C_1 + \frac{C_2}{\lambda} e^{k\lambda \delta} + \alpha S_k \le C_1 + \frac{C_2}{\lambda} e^{k\lambda \delta} +
        \alpha \sum_{i=0}^{k-1}(C_1 +\frac{C_2}{\lambda} e^{i\lambda \delta})(1+\alpha)^{n-1-i}
\end{aligned}
\end{equation}
and then
\begin{equation}
\begin{aligned}
a_k &\le C_1e^{-k\lambda \delta} + \frac{C_2}{\lambda} + \alpha e^{-k\lambda \delta}\sum_{i=0}^{k-1}(C_1 +\frac{C_2}{\lambda} e^{i\lambda \delta})(1+\alpha)^{n-1-i} \\
    &\le C_1e^{-k\lambda \delta} + \frac{C_2}{\lambda} + C_1\alpha e^{-k\lambda \delta}\sum_{i=0}^{k-1}(1+\alpha)^{n-1-i} \\
    &\quad + \alpha \frac{C_2}{\lambda}e^{-k\lambda \delta}\sum_{i=0}^{k-1} e^{i\lambda \delta}(1+\alpha)^{n-1-i} \\
    &= C_1 e^{-k\lambda \delta}(1+\alpha)^k+\frac{C_2}{\lambda} \\
    &\quad +\alpha \frac{C_2}{\lambda}e^{-k\lambda \delta}\frac{(1+\alpha)^k-e^{k\lambda\delta}}{1+\alpha-e^{\lambda\delta}}.
\end{aligned}
\end{equation}
Note that $\alpha$ is the high order infinitesimal of $\delta$ as $\delta\to0$. So we can choose $\delta$ sufficiently small such that $e^{-\lambda \delta}(1+\alpha)<1$, it follows that
$$\aligned a_k&\le C_1 [e^{-\lambda \delta}(1+\alpha)]^k+\frac{C_2}{\lambda} +\alpha \frac{C_2}{\lambda}\frac{[e^{-\lambda \delta}(1+\alpha)]^k-1}{1+\alpha-e^{\lambda\delta}}\\&
\le C_1 +\frac{C_2}{\lambda} +\alpha \frac{C_2}{\lambda}\frac{1}{e^{\lambda\delta}-1-\alpha}\\&
=C_1 +\frac{C_2}{\lambda}\frac{e^{\lambda\delta}-1}{e^{\lambda\delta}-1-\alpha}<\infty.\endaligned$$

For general $t>0$, suppose $t\in(k\delta,(k+1)\delta)$. By (\ref{l10}) we have
$$\mathbb{E}\left|x(t)\right|^p\le C_1e^{-\lambda t} + \frac{C_2}{\lambda} + \alpha e^{-\lambda t}\sum_{i=0}^{k}e^{i\lambda \delta}a_i.$$

Denote $M:=\sup_{k\ge0}a_k.$ We have proved that $M<\infty.$ Now
$$\aligned \mathbb{E}\left|x(t)\right|^p&\le C_1e^{-\lambda t} + \frac{C_2}{\lambda} + M\alpha e^{-\lambda t}\sum_{i=0}^{k}e^{i\lambda \delta}\\&
=C_1e^{-\lambda t} + \frac{C_2}{\lambda} + M\alpha e^{-\lambda t}\frac{e^{(k+1)\lambda\delta}-1}{e^{\lambda\delta}-1}\\&
\le C_1e^{-\lambda t} + \frac{C_2}{\lambda} + M\alpha \frac{e^{\lambda\delta}}{e^{\lambda\delta}-1}
<\infty.\endaligned$$

We complete the proof. \hfill$\Box$

\textbf{Proof of Theorem \ref{youjie1}.} The controlled system (\ref{jsde}) is in fact a hybrid stochastic differential equation with a bounded variable delay. Fix $\xi \in C\left([-\tau, 0]; \mathbb{R}^n\right)$ and $r_0 \in S$. Since both $f$ and $g$ are locally Lipschitz continuous with respect to the first two spacial variables $x$ and $y$, and $u$ is global Lipschitz continuous with respect to $x$, then exists a unique solution for any $t \in [0, \sigma_e)$, where $\sigma_e$ is the lifetime (or explosion time) of the solution. We first prove there exist a unique solution on $[0,\theta)$.

Define
\[
\sigma_k= \inf\left\{ t \in [0, \sigma_e) : |x(t)| \ge k \right\}.
\]
 It is obvious that $\sigma_k$ is non-decreasing as $k$ increases. Denote $\sigma_\infty = \lim_{k \to \infty} \sigma_k$. Then $\sigma_\infty \le \sigma_e$. Therefore, we only need to prove $\sigma_\infty = \infty$.

For $t\in[0,\theta)$, It\^o's formula yields
\begin{equation}\label{eq:expect_x_p}
\begin{aligned}
 \mathbb{E}\left|x(t \wedge \sigma_{k})\right|^q - \left|\xi(0)\right|^q
&\le q\mathbb{E}\int_0^{t \wedge \sigma_{k}} \left|x(s)\right|^{q-2}[
    x^T(s)f(x(s),x_h(s),r(s),s) \\
&\quad +\frac{q-1}{2}\left|g(x(s),x_h(s),r(s),s)\right|^2 ] ds  \\
&\quad +q\mathbb{E}\int_0^{t \wedge \sigma_{k}} \left|x(s)\right|^{q-1}
    \left|u(x(v(s)),r(v(s)),s)\right| ds.
\end{aligned}
\end{equation}

Then by Assumptions \ref{one}, \ref{Khas} and \ref{u1}, we have
\begin{equation}\label{eq:expect_x_p1}
\begin{aligned}
\mathbb{E}\left|x(t \wedge \sigma_{k})\right|^q &\le \left|\xi(0)\right|^q+ Kq\mathbb{E}\int_0^{t \wedge \sigma_{k}}\left|x(s)\right|^{q}ds +Kq\mathbb{E}\int_0^{t \wedge \sigma_{k}}\left|x(s)\right|^{q-2}\left|x_h(s)\right|^{2}ds  \\
&\quad- \alpha_1q\mathbb{E}\int_0^{t \wedge \sigma_{k}}\left|x(s)\right|^{p+q-2}ds +\alpha_2q\mathbb{E}\int_0^{t \wedge \sigma_{k}}\left|x(s)\right|^{p-2}\left|x_h(s)\right|^{q}ds  \\
&\quad +qL\mathbb{E}\int_0^{t \wedge \sigma_{k}}\left|x(s)\right|^{q-1}|x(v(s))|ds.
\end{aligned}
\end{equation}
By Young's inequality and the same idea of \cite{DM}, Theorem 2.4,  we can get
\begin{equation*}
\begin{aligned}
&\quad\mathbb{E}\int_0^{t \wedge \sigma_{k}} |x(s)|^{q-2} |x_h(s)|^2 ds\le \frac{q - 2}{q}\mathbb{E}\int_0^{t \wedge \sigma_{k}} |x(s)|^q ds +\frac{2}{q}\mathbb{E}\int_{0}^{t \wedge \sigma_k} |x_h(s)|^q ds
\end{aligned}
\end{equation*}
\begin{equation*}
\begin{aligned}
&\quad\mathbb{E}\int_0^{t \wedge \sigma_{k}} |x(s)|^{q-1} |x(v(s))|ds\le \frac{q - 1}{q}\mathbb{E}\int_0^{t \wedge \sigma_{k}} |x(s)|^q ds+\frac{1}{q}\mathbb{E}\int_{0}^{t \wedge \sigma_k} |x(v(s))|^q ds
\end{aligned}
\end{equation*}
and
\begin{equation*}
\begin{aligned}
\mathbb{E}\int_0^{t \wedge \sigma_{k}} |x(s)|^{q - 2}|x_h(s)|^p ds
&\le \frac{\alpha_1}{2\alpha_2}\mathbb{E}\int_0^{t \wedge \sigma_{k}} |x(s)|^{p + q - 2} ds +\frac{\alpha_3}{\alpha_2}\mathbb{E}\int_{-\tau}^{t \wedge \sigma_{k}} |x_h(s)|^{p + q - 2} ds
\end{aligned}
\end{equation*}
where
$$\alpha_3:=\frac{q}{p+q-2}\alpha_2^\frac{p+q-2}{q}\left(\frac{2(p-2)}{\alpha_1(p+q-2)}\right)^\frac{p-2}{q}.$$

Substituting the above formulas into (\ref{eq:expect_x_p1}), we have
\begin{equation}\label{l8}
\begin{aligned}
\mathbb{E}|x(t \wedge \sigma_{k})|^q+\frac{\alpha_1}{2}\mathbb{E}\int_0^{t \wedge \sigma_{k}} |x(s)|^{p+q-2} ds
&\le C_{1,t} +(2K+L)(q-1)\mathbb{E}\int_0^{t \wedge \sigma_{k}} |x(s)|^q ds \\
&\quad+ L\mathbb{E}\int_0^{t \wedge \sigma_{k}} |x(v(s))|^q ds.
\end{aligned}
\end{equation}
where
\begin{equation}\label{l9}
\begin{aligned}
C_{1,t} &= \|\xi\|^q +2L\mathbb{E}\int_0^{t} |x(s-h(s))|^q ds +\alpha_3\mathbb{E}\int_0^{t} |x(s-h(s))|^{p+q-2} ds.
\end{aligned}
\end{equation}

It is obvious that $-\tau\le t-h(t)\le0$ for $t\in[0,h'],$ and $C_{1,t}\le C_{1,h'}<\infty$.

Taking supremum on the both sides of (\ref{l8}), we have
\begin{equation}\label{10}
\begin{aligned}
&\quad\sup_{0\le s\le t}\mathbb{E}\left| x(s\wedge \sigma_{k}) \right|^q+\frac{\alpha_1}{2}\mathbb{E}\int_0^{t \wedge \sigma_{k}} |x(s)|^{p+q-2} ds\le C_{1,t} +C_2\int_0^{t}\sup_{0\le r\le s}\mathbb{E}\left| x(s\wedge \sigma_{k}) \right|^q ds.
\end{aligned}
\end{equation}
where $C_2=(2K+L)(q-1)+L.$

An application of the Gronwall inequality yields
\begin{equation}
\label{eq:sup_x_p_bound}
\sup_{0 \le r \le t} \mathbb{E}|x(r \wedge \sigma_{k})|^q \le C_{1,h'}e^{C_2h'}, \ t\in[0,h'].
\end{equation}
Consequently
$${k^q}P({\sigma _{k}} \le h') \le \mathop {\sup }\limits_{0 \le r \le h'} \mathbb{E}|x(r \wedge {\sigma _{k}}){|^q} < \infty.$$

Letting $\delta \to 0$ yields $P\left( {{\sigma _\infty } \le h'} \right) = 0$.
So $ {\sigma_\infty}\ge h'$ almost surely.
Then let $k \to \infty $ in (\ref{eq:sup_x_p_bound}), it follows that
\begin{equation*}
\begin{aligned}
\sup_{0 \le r \le h'} \mathbb{E}|x(r)|^q < \infty.
\end{aligned}
\end{equation*}
Then by letting $k\to \infty$ in (\ref{10}), we have
$$\mathbb{E}\int_0^{h'} |x(s)|^{p+q-2} ds<\infty.$$

Now for $t\in[0,2h'],$ $-\tau\le t-h(t)\le h'$. Then by Lemma \ref{lemma} and the above proof, we have
\begin{equation*}
\begin{aligned}
\mathbb{E}\int_0^{2h'} |x(s)|^{p+q-2} ds
&\le h^*\mathbb{E}\int_{-\tau}^{h'} |x(s)|^{p+q-2} ds \le h^*\tau\|\xi\|^{p+q-2}+h^*\mathbb{E}\int_0^{h'} |x(s)|^{p+q-2} ds < \infty.
\end{aligned}
\end{equation*}

Note that (\ref{10}) still holds for $t\in[0,2h']$. Then we can prove $ {\sigma_\infty}\ge 2h'$ almost surely. Then
\begin{equation*}
\begin{aligned}
\sup_{0 \le r \le 2h'} \mathbb{E}|x(r)|^q < \infty.
\end{aligned}
\end{equation*}
and
$$\mathbb{E}\int_0^{2h'} |x(s)|^{p+q-2} ds<\infty.$$
Repeating this procedure, we can prove that $ \sigma_e\ge{\sigma_\infty}>\theta$ almost surely. That is, there exists a unique solution on $[0,\theta]$.

For $t\in[\theta, T],$ there is no control here, the equation becomes to (\ref{sde}) with the initial value $x_\theta(s)=x(\theta-s), s\in[-\tau,0]$. Note that
(\ref{sde}) has unique local solution on $[\theta,\sigma_e)$. Then by the same method used above, it follows that $ \sigma_e\ge{\sigma_\infty}>T.$ So there exists a unique solution on $[0,T]$.

Repeating this procedure, it follows that there exists a unique solution on $[T,2T]$, $[2T,3T]$, $\cdots$, etc. That is, the controlled system (\ref{jsde}) with any initial value $x_{0}=\xi$ has a unique global solution $x(t)$ on $t\geqslant 0$.

Now let us prove (\ref{youjie112}).

For any $\lambda>0$, It\^o's formula yields
\begin{equation*}
\begin{aligned}
 e^{\lambda t}|x(t)|^{q}
&\le \|\xi\|^{q} + \int_0^te^{\lambda s}|x(s)|^{q-2}[(\lambda+qK)|x(s)|^2  +qK|x_h(s)|^2 \\&- q\alpha_1|x(s)|^p + q\alpha_2|x_h(s)|^p] ds +q\int_0^te^{\lambda s}|x(s)|^{q-1}|u(x(v(s)))|ds \\
&\le \|\xi\|^{q} + (\lambda+2K(q-1)+L(q-1)\delta^{-\frac{1}{q-1}})\mathbb{E}\int_{0}^{t}e^{\lambda s}\left|x(s)\right|^qds  \\
&\quad +2K\mathbb{E}\int_{0}^{t}e^{\lambda s}\left|x_h(s)\right|^qds-\bar{\alpha}_1\mathbb{E}\int_{0}^{t}e^{\lambda s}\left|x(s)\right|^{p+q-2}ds \\
&\quad+\bar{\alpha}_2\mathbb{E}\int_{0}^{t}e^{\lambda s}\left|x_h(s)\right|^{p+q-2}ds + L\delta\mathbb{E}\int_{0}^{t}e^{\lambda s}|x(v(s))|^qds.
\end{aligned}
\end{equation*}
 where
  $$\bar{\alpha}_1 =q\alpha_1-\frac{\alpha_2q(q-2)}{p+q-2},\quad
\bar{\alpha}_2 =\frac{\alpha_2pq}{p+q-2}.$$

It follows that
\begin{equation*}
\begin{aligned}
e^{\lambda t} \mathbb{E}\left|x(t)\right|^q
&\le \|\xi\|^q + \mathbb{E}\int_{0}^{t}e^{\lambda s}[C(q,\delta) -(\bar{\alpha}_1-\lambda)\left|x(s)\right|^{p+q-2}]ds  \\
&\quad +(\bar{\alpha}_2+\lambda)\mathbb{E}\int_{0}^{t}e^{\lambda s}\left|x_h(s)\right|^{p+q-2}ds  +L\delta\mathbb{E}\int_{0}^{t}e^{\lambda s}|x(v(s))|^qds.
\end{aligned}
\end{equation*}
where
\begin{equation*}
\begin{aligned}
C(q,\delta) &:= 2\sup_{u\ge 0}[(\lambda+2K(q-1)+L(q-1)\delta^{-\frac{1}{q-1}})u^q -\lambda u^{p+q-2}]  < \infty.
\end{aligned}
\end{equation*}
for any fixed $\lambda>0$ and $\delta>0.$

Since $\bar{\alpha}_1>\bar{\alpha}_2h^*$, we can choose $\lambda$ sufficiently small such that
$$\bar{\alpha}_1-\lambda=h^*e^{\lambda\tau}(\bar{\alpha}_2+\lambda).$$

Then we have
\begin{equation}\label{l10'}
\begin{aligned}
e^{\lambda t} \mathbb{E}\left|x(t)\right|^q
&\le \|\xi\|^q + (\bar{\alpha}_2+\lambda)h^*\tau e^{\lambda\tau}\|\xi\|^{p+q-2}   +C(q,\delta)\frac{e^{\lambda t}}{\lambda} + L\delta\int_{0}^{t}e^{\lambda s}\mathbb{E}|x(v(s))|^qds.
\end{aligned}
\end{equation}

Lemma \ref{l1} implies that
$$\sup_{t\ge0}\mathbb{E}\left|x(t)\right|^q<\infty.$$

We complete the proof. \hfill$\Box$

\section{Main results}
	
We have just shown that the $q$-th moment of the controlled system $\left(\ref{jsde} \right)$ is bounded as long as the coefficients $f$ and $g$ satisfy Assumption \ref{Khas} and the control function $u$ satisfies Assumption \ref{u1}. However, such a control may not be able to stabilize the given SDDE (\ref{sde}). So we need more carefully design the control function in order to stabilize the controlled system $\left(\ref{jsde}\right)$. In this section, we will design the control function to meet a number of conditions under our standing Assumptions \ref{one}, \ref{delay}, \ref{Khas}, and \ref{u1}, and then show such designed control function will indeed guarantee the exponential stability of the controlled system $\left(\ref{jsde}\right)$. Let us begin to state our conditions.
	\begin{Condition}\label{c1}
		Design the control function $u:\mathbb{R}^n\times S \times \mathbb{R}_{+}\to\mathbb{R}^n$ such that there are constants $\alpha_{ji} > 0$ , $\beta_{ji}> 0$ and
		$l_{ji} > 0, k_{ji}\in R, i\in S, j=1,2$ for both
		\begin{equation}\label{c11}
 \aligned &\quad x^{T}[f(x,y,i,t)+u(x,i,t)]+\dfrac{1}{2}|g(x,y,i,t)|^{2}
			\le k_{1i}|x|^2+l_{1i}|y|^{2}-\beta_{1i}|x|^p+\alpha_{1i}|y|^{p} \endaligned
		\end{equation}
and
\begin{equation}\label{c12}
\aligned &\quad x^{T}[f(x,y,i,t)+u(x,i,t)]+\dfrac{q_1}{2}|g(x,y,i,t)|^{2}
			\le k_{2i}|x|^2+l_{2i}|y|^{2}-\beta_{2i}|x|^p+\alpha_{2i}|y|^{p}\endaligned
		\end{equation}
 for any $t\in[nT,nT+\theta)$. 	Moreover,
		\begin{equation}\label{c13}
			\begin{aligned}
				&\mathcal{A}_{1}:=-2\textrm{diag}\left(k_{11},\cdots,k_{1N} \right) -\Gamma \\
				&\mathcal{A}_{2}:=-(q_{1}+1)\textrm{diag}\left(k_{21},\cdots,k_{2N} \right) -\Gamma\\
			\end{aligned}
		\end{equation}
		are nonsingular M-matrices.
	\end{Condition}
Denote
\begin{equation*}
\begin{aligned}
&(\theta_1,\cdots,\theta_N)^T=\mathcal{A}_{1}^{-1}(1,\cdots,1)^T,\quad (\bar{\theta}_1,\cdots,\bar{\theta}_N)^T=\mathcal{A}_{2}^{-1}(1,\cdots,1)^T\\
&\zeta_1=2\max_{i\in S}\theta_i l_{1i},\quad \zeta_2=2\min_{i\in S}\theta_i \beta_{1i},\quad \zeta_3=2\max_{i\in S}\theta_i \alpha_{1i},\\& \zeta_4=(q_1+1)\max_{i\in S}\bar{\theta}_i l_{2i},\quad \zeta_5=(q_1+1)\min_{i\in S}\bar{\theta}_i \beta_{2i},\quad \zeta_6=(q_1+1)\max_{i\in S}\bar{\theta}_i \alpha_{2i}.
\end{aligned}
\end{equation*}

We can choose $\zeta_i, i=1,\cdots,6$ such that
$$1>h^*\zeta_1,\quad \zeta_2>h^*\zeta_3$$
and
$$1>\frac{\zeta_4(q_1-1+2h^*)}{q_1+1},\quad \zeta_5>\zeta_6\frac{q_1-1+ph^*}{p+q_1-1}.$$

It is obvious that $\theta_i, \bar{\theta}_i, i=1,\cdots,N$ are positive.

Define
\begin{equation}\label{test}
U(x,i)=\theta_i|x|^2+\bar{\theta}_i|x|^{q_1+1}, (x,i)\in\mathbb{R}^n\times S,\end{equation}
an operator $L$ for $t\in[nT,nT+\theta)$ such at
\begin{equation}\label{LU}
\begin{aligned}
LU(x,y,i,t)
&= 2\theta_i(x^T[f(x,y,i,t)+u(x,i,t)]+\dfrac{1}{2}|g(x,y,i,t)|^{2})\\
&\quad +(q_1+1)\bar{\theta}_i |x|^{q_1-1} (x^{T}[f(x,y,i,t)+u(x,i,t)] +\dfrac{q_1}{2}|g(x,y,i,t)|^{2})  \\
&\quad +\sum_{j=1}^N\gamma_{ij}(\theta_j|x|^2+\bar{\theta}_j|x|^{q_1+1})
\end{aligned}
\end{equation}
and $L'$ for $t\in[nT+\theta,(n+1)T)$ such that
\begin{equation}\label{L'U}
\begin{aligned}
L'U(x,y,i,t)
&= 2\theta_i(x^Tf(x,y,i,t)+\dfrac{1}{2}|g(x,y,i,t)|^{2})  \\
&\quad +(q_1+1)\bar{\theta}_i |x|^{q_1-1}(x^{T}f(x,y,i,t) +\dfrac{q_1}{2}|g(x,y,i,t)|^{2})  \\
&\quad +\sum_{j=1}^N\gamma_{ij}(\theta_j|x|^2+\bar{\theta}_j|x|^{q_1+1}).
\end{aligned}
\end{equation}

Then it follows that
\begin{equation}\label{suanzi}
\begin{aligned}
 LU(x,y,i,t)
&\le 2\theta_i(k_{1i}|x|^2+l_{1i}|y|^{2}-\beta_{1i}|x|^p+\alpha_{1i}|y|^{p})  \\
&\quad +(q_1+1)\bar{\theta}_i|x|^{q_1-1}(k_{2i}|x|^2+l_{2i}|y|^{2}-\beta_{2i}|x|^p+\alpha_{2i}|y|^{p})  \\
&\quad +  \sum_{j=1}^N\gamma_{ij}(\theta_j|x|^2+\bar{\theta}_j|x|^{q_1+1}) \\
&\le -|x|^2+\zeta_1|y|^2-\zeta_2|x|^p+\zeta_3|y|^p-|x|^{q_1+1}  \\
&\quad +\zeta_4|x|^{q_1-1}|y|^2-\zeta_5|x|^{p+q_1-1} +\zeta_6|x|^{q_1-1}|y|^p \\
&\le -|x|^2+\zeta_1|y|^2-\zeta_2|x|^p+\zeta_3|y|^p  \\
&\quad +\frac{2\zeta_4}{q_1+1}|y|^{q_1+1}-(1-\frac{\zeta_4(q_1-1)}{q_1+1})|x|^{q_1+1}  \\
&\quad +\frac{p\zeta_6}{p+q_1-1}|y|^{p+q_1-1}-(\zeta_5-\frac{\zeta_6(q_1-1)}{p+q_1-1})|x|^{p+q_1-1}.
\end{aligned}
\end{equation}

\begin{Remark}
If $u$ in the controlled system (\ref{jsde}) is instant (i.e. $u=u(x(t),r(t),t)$) and there is no intermittent (i.e. $I(t)\equiv 1$), then as interpreted in \cite{DM}, Condition \ref{c1} implies that the instant and nointermittent controlled system is exponentially stable. However, since we consider intermittent feedback controls based on discrete-time observations, another condition is needed.
\end{Remark}

\begin{Condition}\label{c2}
Find positive constants $ \gamma_{i}, i=1,\cdots,8$ and $\gamma'_j, j=4,5,6$ with $1\wedge\gamma_4>2(\gamma_5\vee\gamma_6\vee\gamma'_5\vee\gamma'_6) h^*=:\bar{\gamma}h^*$ and $W\in C(\mathbb{R}^n,\mathbb{R}_+)$ such that
\begin{equation}\label{c21}
\left\{
\begin{aligned}
&\quad LU(x,y,i,t)
+\gamma_{1}(2\theta_{i}|x|+(q_{1}+1)\bar{\theta}_{i}|x|^{q_{1}})^2 + \gamma_{2}|f(x,y,i,t)|^2 + \gamma_{3}|g(x,y,i,t)|^2 \\
&\le -\gamma_{4}|x|^{2}+\gamma_5|y|^2-W(x)+\gamma_6 W(y), \qquad t\in[nT,nT+\theta), \\
&L'U(x,y,i,t)\le \gamma'_{4}|x|^{2}+\gamma'_5|y|^2-W(x)+\gamma'_6 W(y), \qquad t\in[nT+\theta,(n+1)T)
\end{aligned}
\right.
\end{equation}
and
\begin{equation}\label{c21'}
\gamma_7|x|^{q_1+p-1}\le W(x)\le \gamma_8(|x|^2+|x|^{q_1+p-1})
		\end{equation}
		for all $\left( x,y,i,t\right) \in R^{n}\times R^{n}\times S\times R_{+}$.
	\end{Condition}

It is easy to verify that Condition \ref{c2} can always holds for suitable constants $\gamma_i$ and $\gamma'_j$ since $u$ satisfies linear growth condition (\ref{2.10}).

\begin{Condition}\label{c3}
		Denote
$$C(\gamma,L):=\sqrt{(\min\limits_{i\in S}{\gamma_{ii}})^2+\frac{\gamma_1}{2}\bigwedge\frac{\gamma_1^2(\gamma_4-\gamma_5)}{L^2}}.$$ Let the observation time interval $\delta$ satisfy
		\begin{equation}\label{c31}
\begin{aligned}
0<\delta&<\frac{\sqrt{\gamma_{1}\gamma_{2}}}{2L}\bigwedge \frac{\gamma_{1}\gamma_{3}}{2L^{2}} \bigwedge \frac{\min\limits_{i\in S}{\gamma_{ii}}+C(\gamma,L)}{2\gamma_1}.
\end{aligned}
\end{equation}
\end{Condition}

Denote
\begin{equation}\label{c33}
		a_1:=\min_{i\in S}\theta_i,	a_{2}=\max_{i \in S}\theta_{i},a_{3}=\max_{i \in S}\bar \theta_{i}.
		\end{equation}

Now we are in the position to state our main result.

	\begin{Theorem}\label{dl1}
		Under the same conditions of Theorem \ref{youjie1}, we can design the control function $u$ to satisfy Condition \ref{c1} and then choose positive constants $ \gamma_{i}, i=1,\cdots,8$ and $\gamma'_j, j=4,5,6$ to meet Condition \ref{c2}. If we further make sure Condition \ref{c3} to hold, then we can choose $\theta\in((1-\frac{a_1\varepsilon}{C_5})T+\frac{a_1\log (1+\frac{2CL^{2}\delta}{C_4\gamma_{1}})}{C_5},T]$ (where $\varepsilon$, $C_4$ and $C_5$ will be determined in the proof) such that the solution of the controlled system $ \left(\ref{jsde} \right)$ has the property that for any $\bar{q}\in \left[ 2,q\right)$ ($q$ satisfies (\ref{pq}) in Assumption \ref{Khas}) and any initial value $ x_0 = \xi \in C([-\tau,0];R^{n})$,
		\begin{equation}\label{dl11}
			\limsup_{t\to\infty}\frac{1}{t}\log(\mathbb{E}|x(t)|^{\bar{q}})<0.
		\end{equation}
		That is, the controlled system $ \left(\ref{jsde} \right)  $ is exponentially stable in $ L^{\bar q}.$
	\end{Theorem}
	\textbf{Proof.} Let $ U $ as in (\ref{test})
	while define $LU$ as in (\ref{suanzi}),	and define two  segments,
	$ \hat{x}_{t}:=\left\lbrace x\left( t+s\right) :-2\tau\leq s\leq 0\right\rbrace  $ and
	$ \hat{r}_{t}:=\left\lbrace r\left( t+s\right) :-2\tau\leq s\leq 0\right\rbrace  $  for $ t\geqslant0 $. For $  \hat{x}_{t} $ and $\hat{r}_{t}  $ to be well defined for $ 0\leq t<2\tau $. We set $ x(s) \equiv x(0) $ and $ r(s) \equiv r_{0} $ for $ s \in \left[ -2\tau,0\right).$
Motivated by \cite{HRZ}, the Lyapunov functional used in this proof is of the form
	\begin{equation}\label{la}
		V(\hat{x}_{t},\hat{r}_{t},t)=U(x(t),r(t))+cI(t)\int^{0}_{-\delta}\int^{t}_{t+s}\eta(l)dlds
	\end{equation}
	and
\begin{equation}
\begin{aligned}
\eta\left(l\right)
&= \delta\left|f\left(x(l),x_h(l),r\left(l\right),l\right)+I(l)u\left(x(v(l)),r\left(l\right),l\right)\right| ^{2} +\left|g\left(x(l),x_h(l),r\left(l\right),l\right)\right|^{2}
\end{aligned}
\end{equation}
	where $ U $ has been defined  and $ c $ is a positive
	constant to be determined later while we set
	\begin{equation*}
\begin{aligned}
&f\left( x,y,i,l\right)=f\left(x,y,i,0\right),\quad g\left(x,y,i,l\right)=g\left(x,y,i,0\right), \quad u\left(x,i,l\right)=u\left(x,i,0\right)
\end{aligned}
\end{equation*}
	for $ \left(x,y,i,v\right) \in R^{n}\times R^{n}\times S\times\left[ -2\tau,0\right)$.

By the generalised It\^o formula (see, e.g., \cite{jiaocai}), we have
	\begin{equation}\label{35}
\begin{aligned}
&\quad dU(x(t),r(t))= \mathcal{L}U(x(t),x_h(t),x(v(t)),r(t),r(v(t)),t)dt +dM(t)
\end{aligned}
\end{equation}
	for $ t\geqslant0 $, where $ M\left( t\right)  $ is the local martingale term with
	$$ \aligned M\left( t\right)&=\int_0^t U_x(x(s),r(s))g(x(s),x_h(s),r(s),s)dB(s)\\&\quad+\int_0^t\sum_{j\neq r(s-)}(U(x(s),j)-U(x(s),r(s-)))dN_{r,j}(s)\endaligned$$
where $N_{r,j}(s)$ is the compensated martingale measure,
and $\mathcal{L}U  $ is defined by
	\begin{equation}\label{35.5}
        \begin{aligned}
        \quad\mathcal{L}U(x,y,z,i,j,t)
        &=2\theta_{i}[x^{T}[f(x,y,i,t)+I(t)u(z,j,t)]+\frac{1}{2}|g(x,y,i,t)|^{2}] \\
        &\quad+ (q_{1}+1)\bar{\theta}_{i}|x|^{q_{1}-1}[x^{T}[f(x,y,i,t)+I(t)u(z,j,t)] +\frac{1}{2}|g(x,y,i,t)|^{2}]  \\
        &\quad +\frac{(q_{1}+1)(q_{1}-1)}{2}\bar{\theta}_{i}|x|^{q_{1}-3}|x^{T}g(x,y,i,t)|^{2}  \\
        &\quad +\sum_{j=1}^{N}\gamma_{ij}(\theta_{j}|x|^{2}+\bar{\theta}_{j}|x|^{q_{1}+1}).
        \end{aligned}
        \end{equation}

On the other hand, it is obvious that
	\begin{equation}\label{36}
		d\left(\int^{0}_{-\delta}\int^{t}_{t+s}\eta(l)dlds\right)=\left(\delta \eta(t)-\int^{t}_{t-\delta}\eta(s)ds\right)dt
	\end{equation}

Note that $V$ defined above is discontinuous at $t=nT, nT+\theta$ for any $ n=0,1,2,\cdots.$ So in this case, The It\^o's formula becomes to
$$\aligned V(\hat{x}_{t},\hat{r}_{t},t)&=V(\hat{x}_{0},\hat{r}_{0},0)+\int_0^t\mathcal{L}U(x(s),x_h(s),x(v(s)),r(s),r(v(s)),s)ds+M(t)\\&\quad
+c\sum_{0<s\le t}\Delta I(s)\int^{0}_{-\delta}\int^{s}_{s+r}\eta(l)dldr +c\int_0^tI(s)\left(\delta \eta(s)-\int^{s}_{s-\delta}\eta(l)dl\right)ds,\endaligned$$
where $\Delta I(s)=I(s)-I(s-)$.

Furthermore, it is easy to see that for $t\in[nT,nT+\theta),$
	\begin{equation}\label{37.5}
        \begin{aligned}
        &\quad\mathcal{L}U(x,y,z,i,j,t)=\mathit{L}U(x,y,i,t) +[2\theta_{i}+(q_{1}+1)\bar{\theta}_{i}|x|^{q_{1}-1}]x^{T}[u(x,i,t)-u(z,j,t)]I(t)
        \end{aligned}
        \end{equation}
        and  for $t\in[nT+\theta,(n+1)T),$
        \begin{equation}\label{37.5'}
        \begin{aligned}
        &\quad\mathcal{L}U(x,y,z,i,j,t)=\mathit{L'}U(x,y,i,t)
        \end{aligned}
        \end{equation}
	where the functions $\mathit{L}U$ and $\mathit{L'}U$ have been defined by \eqref{LU} and \eqref{L'U}. Then it follows that
	\begin{equation}\label{38}
\begin{aligned}
e^{\varepsilon t}\mathbb{E}V(\hat{x}_{t},\hat{r}_{t},t)&\le\mathbb{E}V(\hat{x}_{0},\hat{r}_{0},0) +\int_{0}^{t}e^{\varepsilon s}\mathbb{E}(\varepsilon V(\hat{x}_{s},\hat{r}_{s},s)+\mathbb{L}V(\hat{x}_{s},\hat{r}_{s},s))ds\\&
\quad+c\mathbb{E}\sum_{0<s\le t}e^{\varepsilon s}\Delta I(s)\int^{0}_{-\delta}\int^{s}_{s+r}\eta(l)dldr
\end{aligned}
\end{equation}
	where
	\begin{equation}\label{39}
\begin{aligned}
\mathbb{L}V(\hat{x}_{t},\hat{r}_{t},t)&=\mathcal{L}U(x(t),x_h(t),r(t),t)  +cI(t)\left(\delta \eta(t)-\int^{t}_{t-\delta}\eta(s)ds\right).
\end{aligned}
\end{equation}

Then for any $t\in[nT,nT+\theta),$ we have
\begin{equation}\label{38'}
\begin{aligned}
e^{\varepsilon t}\mathbb{E}V(\hat{x}_{t},\hat{r}_{t},t)&\le e^{\varepsilon nT}\mathbb{E}V(\hat{x}_{nT},\hat{r}_{nT},nT) +\int_{nT}^{t}e^{\varepsilon s}\mathbb{E}(\varepsilon V(\hat{x}_{s},\hat{r}_{s},s)+\mathbb{L}V(\hat{x}_{s},\hat{r}_{s},s))ds\\&
\quad+c\mathbb{E}\sum_{nT<s\le t}e^{\varepsilon s}\Delta I(s)\int^{0}_{-\delta}\int^{s}_{s+r}\eta(l)dldr
\end{aligned}
\end{equation}
and for any $t\in[nT+\theta,(n+1)T),$
\begin{equation}\label{38''}
\begin{aligned}
e^{\varepsilon t}\mathbb{E}V(\hat{x}_{t},\hat{r}_{t},t)&\le e^{\varepsilon (nT+\theta)}\mathbb{E}V(\hat{x}_{nT+\theta},\hat{r}_{nT+\theta},nT+\theta)\\&
 \quad+\int_{nT+\theta}^{t}e^{\varepsilon s}\mathbb{E}(\varepsilon V(\hat{x}_{s},\hat{r}_{s},s)+\mathbb{L}V(\hat{x}_{s},\hat{r}_{s},s))ds\\&
\quad+c\mathbb{E}\sum_{nT+\theta<s\le t}e^{\varepsilon s}\Delta I(s)\int^{0}_{-\delta}\int^{s}_{s+r}\eta(l)dldr.
\end{aligned}
\end{equation}

Note that in each case, the last term equals to $0$ since $I(s)$ is constant $0$ or $1$ in open intervals $(nT,nT+\theta)$ and $(nT+\theta,(n+1)T)$.
Then for any $t\in[nT,nT+\theta),$ we have
\begin{equation}\label{38'''}
\begin{aligned}
e^{\varepsilon t}\mathbb{E}V(\hat{x}_{t},\hat{r}_{t},t)&\le e^{\varepsilon nT}\mathbb{E}V(\hat{x}_{nT},\hat{r}_{nT},nT) +\int_{nT}^{t}e^{\varepsilon s}\mathbb{E}(\varepsilon V(\hat{x}_{s},\hat{r}_{s},s)+\mathbb{L}V(\hat{x}_{s},\hat{r}_{s},s))ds
\end{aligned}
\end{equation}
and for any $t\in[nT+\theta,(n+1)T),$
\begin{equation}\label{38''''}
\begin{aligned}
e^{\varepsilon t}\mathbb{E}V(\hat{x}_{t},\hat{r}_{t},t)&\le e^{\varepsilon (nT+\theta)}\mathbb{E}V(\hat{x}_{nT+\theta},\hat{r}_{nT+\theta},nT+\theta)\\&
 \quad+\int_{nT+\theta}^{t}e^{\varepsilon s}\mathbb{E}(\varepsilon V(\hat{x}_{s},\hat{r}_{s},s)+\mathbb{L}V(\hat{x}_{s},\hat{r}_{s},s))ds.
\end{aligned}
\end{equation}

Moreover, Theorem \ref{youjie1} and Assumptions \ref{polynomial} and \ref{u1} implies that
	$$	\sup_{0\le t < \infty}\mathbb{E}\left| \mathbb{L}V\left( \hat{x}_{t},\hat{r}_{t},t\right) \right| <\infty.$$
	Let us now estimate $ \mathbb{L}V\left( \hat{x}_{t},\hat{r}_{t},t\right).$

Let $c=\frac{2L^{2}}{\gamma_{1}} $. By Assumption \ref{u1}, we have
	\begin{equation}\label{41}
\begin{aligned}
&\quad\mathbb{E}[2\theta_{r(t)}+(q_{1}+1)\bar{\theta}_{r(t)}|x(t)|^{q_{1}-1}]  x^{T}(t)[u(x(v(t)),r(v(t)),t)-u(x(t),r(t),t)] \\
&\leq \gamma_1\mathbb{E}[2\theta_{r(t)}|x(t)|+(q_{1}+1)\bar{\theta}_{r(t)}|x(t)|^{q_{1}}]^{2}   +\frac{L^{2}}{2\gamma_{1}}\mathbb{E}|x(t)-x(v(t))|^{2}  \\
&\quad +\frac{1}{2\gamma_{1}}\mathbb{E}|u(x(v(t)),r(v(t)),t)-u(x(v(t)),r(t),t)|^{2} \\
&\le \gamma_1 \mathbb{E}[2\theta_{r(t)}|x(t)|+(q_{1}+1)\bar{\theta}_{r(t)}|x(t)|^{q_{1}}]^{2}+\frac{L^{2}}{2\gamma_{1}}\mathbb{E}|x(t)-x(v(t))|^{2}  \\
&\quad +\frac{2L^2}{\gamma_1}(1-\rho(t-v(t)))\mathbb{E}|x(v(t))|^{2}.
\end{aligned}
\end{equation}
where $\rho(s):=e^{\min_{i\in S}{\gamma_{ii}}s}.$ We have used Lemma 2 of \cite{yu2023stability} or Lemma 3 of \cite{SLZM} in the last inequality.

	By Condition (4.3), we also have
	\begin{equation}\label{42}
		\frac{4L^2\delta^{2}}{\gamma_1}\leq\gamma_{2}\quad \textrm{and}\quad \frac{2L^2\delta}{\gamma_1}\leq\gamma_{3}.
	\end{equation}

It then follows from \eqref{39} along with Condition \ref{c2} and	inequality \eqref{2.10} that
	\begin{equation}\label{43}
\begin{aligned}
&\quad\mathbb{E}\mathbb{L}V(\hat{x}_{t},\hat{r}_{t},t)\\
&\leq \mathbb{E}\mathit{L}U(x(t),x_h(t),r(t),t) +\gamma_1 \mathbb{E}[2\theta_{r(t)}|x(t)|+(q_{1}+1)\bar{\theta}_{r(t)}|x(t)|^{q_{1}}]^{2}  \\
&\quad +\frac{L^{2}}{2\gamma_{1}}\mathbb{E}|x(t)-x(v(t))|^{2} -\frac{2L^2\min_{i\in S}{\gamma_{ii}}\delta}{\gamma_{1}}\mathbb{E}|x(v(t))|^{2}  +\frac{2L^2}{\gamma_1}\mathbb{E}\left(\delta \eta(t)-\int^{t}_{t-\delta}\eta(s)ds\right) \\
&\le \mathbb{E}\mathit{L}U(x(t),x_h(t),r(t),t) +\gamma_1\mathbb{E}[2\theta_{r(t)}|x(t)|+(q_{1}+1)\bar{\theta}_{r(t)}|x(t)|^{q_{1}}]^{2}  \\
&\quad +\frac{2L^2}{\gamma_1}\left(2\delta^2|f(x(t),x_h(t),r(t),t)|^2+ \delta|g(x(t),x_h(t),r(t),t)|^2 + 2\delta^2L^2\mathbb{E}|x(v(t))|^{2}\right)  \\
&\quad +\frac{L^{2}}{2\gamma_{1}}\mathbb{E}|x(t)-x(v(t))|^{2} - \frac{2L^2\min_{i\in S}{\gamma_{ii}}\delta}{\gamma_{1}}\mathbb{E}|x(v(t))|^{2}  -\frac{2L^2}{\gamma_1}\mathbb{E}\int^{t}_{t-\delta}\eta(s)ds \\
&\leq -\left(\gamma_{4}-\frac{-4L^2\min_{i\in S}{\gamma_{ii}}\delta}{\gamma_{1}}-4L^2\delta^2\right)\mathbb{E}|x(t)|^{2} +\gamma_{5}\mathbb{E}|x_h(t)|^{2} \\&
\quad- \mathbb{E}W(x(t)) + \gamma_6\mathbb{E}W(x_h(t))  \\
&\quad -\frac{L^2+8L^2\min_{i\in S}{\gamma_{ii}}\delta-8L^2\delta^2\gamma_1}{\gamma_1}\mathbb{E}\int^{t}_{t-\delta}\eta(s)ds.
\end{aligned}
\end{equation}
We have used the fact that
\begin{equation*}
\begin{aligned}
\eta(t)
&\le 2\delta|f(x(t),x_h(t),r(t),t)|^2 +2\delta L^2|x(v(t))|^2 +|g(x(t),x_h(t),r(t),t)|^2
\end{aligned}
\end{equation*}
in the second inequality and
$$\mathbb{E}|x(v(t))|^{2}\le 2\mathbb{E}|x(t)|^{2}+2\mathbb{E}|x(t)-x(v(t))|^{2}$$
$$\mathbb{E}|x(t)-x(v(t))|^{2}\le 2\mathbb{E}\int_{v(t)}^t\eta(s)ds\le 2\mathbb{E}\int_{t-\delta}^t\eta(s)ds$$
in the last inequality.

Moreover, we clearly have
	\begin{equation}\label{66}
\begin{aligned}
\mathbb{E}|x(s)|^{q_{1}+1}
&\leq \mathbb{E}|x(s)|^{2} + \mathbb{E}|x(s)|^{p+q_{1}-1} \le \mathbb{E}|x(s)|^{2} + \frac{1}{\gamma_7}\mathbb{E}W(x(s)).
\end{aligned}
\end{equation}

Denote
\begin{equation}B_1:=\gamma_{4}-\frac{-4L^2\min_{i\in S}{\gamma_{ii}}\delta}{\gamma_{1}}-4L^2\delta^2-\varepsilon(a_2+a_3),\quad B_2=\frac{L^2+8L^2\min_{i\in S}{\gamma_{ii}}\delta-8L^2\delta^2\gamma_1}{\gamma_1}.\end{equation}
Then for any $\varepsilon>0$ and $t\in[nT,nT+\theta),$ we have

\begin{equation}\label{43'}
\begin{aligned}
&\quad\mathbb{E}(\varepsilon V(\hat{x}_{t},\hat{r}_{t},t)+\mathbb{L}V(\hat{x}_{t},\hat{r}_{t},t))\\
&\leq \varepsilon \frac{2L^2}{\gamma_1}\int^{0}_{-\delta}\int^{t}_{t+s}\eta(l)dlds -B_1\mathbb{E}|x(t)|^{2} +\gamma_{5}\mathbb{E}|x_h(t)|^{2} \\
&\quad -(1-\frac{\varepsilon a_3}{\gamma_7})\mathbb{E}W(x(t)) + \gamma_6\mathbb{E}W(x_h(t))-B_2\mathbb{E}\int^{t}_{t-\delta}\eta(s)ds \\
&\le -B_1\mathbb{E}|x(t)|^{2} +\gamma_{5}\mathbb{E}|x_h(t)|^{2}- (1-\frac{\varepsilon a_3}{\gamma_7})\mathbb{E}W(x(t)) + \gamma_6\mathbb{E}W(x_h(t))  \\
&\quad -\left(B_2- \frac{2L^2\varepsilon\delta}{\gamma_1}\right)\mathbb{E}\int^{t}_{t-\delta}\eta(s)ds.
\end{aligned}
\end{equation}

Note that (\ref{c31}) implies
$$L^2+8L^2\min_{i\in S}{\gamma_{ii}}\delta-8L^2\delta^2\gamma_1>0.$$

Then we can choose $0<\varepsilon\le \frac{L^2+8L^2\min_{i\in S}{\gamma_{ii}}\delta-8L^2\delta^2\gamma_1}{2L^2\delta}$ such that
$-\left(B_2- \frac{2L^2\varepsilon\delta}{\gamma_1}\right)\le 0.$

Then we have
\begin{equation}\label{43''}
\begin{aligned}
 \mathbb{E}(\varepsilon V(\hat{x}_{t},\hat{r}_{t},t)+\mathbb{L}V(\hat{x}_{t},\hat{r}_{t},t))
&\le -B_1\mathbb{E}|x(t)|^{2} +\gamma_{5}\mathbb{E}|x_h(t)|^{2} \\
&\quad - (1-\frac{\varepsilon a_3}{\gamma_7})\mathbb{E}W(x(t)) + \gamma_6\mathbb{E}W(x_h(t)).
\end{aligned}
\end{equation}
We have used the fact that
$$\int^{0}_{-\delta}\int^{t}_{t+s}\eta(l)dlds\le\delta\int^{t}_{t-\delta}\eta(s)ds$$
in (\ref{43'}).

Then by (\ref{38'''}), for any $t\in[nT,nT+\theta),$
\begin{equation}\label{44}
\begin{aligned}
&\quad\mathbb{E}e^{\varepsilon t}V(\hat{x}_{t},\hat{r}_{t},t)\\
&\le e^{\varepsilon nT}\mathbb{E}V(\hat{x}_{nT},\hat{r}_{nT},nT)
 + (\gamma_5\vee\gamma_6)h^*e^{\varepsilon\tau}\int_{nT-\tau}^{nT}e^{\varepsilon s}(\mathbb{E}|x(s)|^{2} +\mathbb{E}W(x(s)))ds \\
&\quad- \left(B_1- \gamma_{5}h^*e^{\varepsilon \tau}\right)\int_{nT}^{t}e^{\varepsilon s}\mathbb{E}|x(s)|^{2}ds -\left(1-\gamma_{6}h^*e^{\varepsilon \tau}-\frac{\varepsilon a_{3}}{\gamma_7}\right)\int_{nT}^{t}e^{\varepsilon s}\mathbb{E}W(x(s))ds.
\end{aligned}
\end{equation}

On the other hand, for $t\in[nT+\theta,(n+1)T),$ we have
\begin{equation}\label{43'''}
\begin{aligned}
\mathbb{E}(\varepsilon V(\hat{x}_{t},\hat{r}_{t},t)+\mathbb{L}V(\hat{x}_{t},\hat{r}_{t},t))
&\le \mathbb{E}(\varepsilon U(x_{t},r_{t})+L'U(x(t),x_h(t),r(t),t)) \\
&\le \left(\gamma'_{4}+\varepsilon(a_2+a_3)\right)\mathbb{E}|x(t)|^{2} + \gamma'_{5}\mathbb{E}|x_h(t)|^{2} \\
&\quad -(1-\frac{\varepsilon a_3}{\gamma_7})\mathbb{E}W(x(t)) + \gamma'_6\mathbb{E}W(x_h(t)).
\end{aligned}
\end{equation}

Let $\theta=m\tau, T=(m+m')\tau$. Denote
$$D_{i}=\int_{(i-1)\tau}^{i\tau}e^{\varepsilon s}(\mathbb{E}|x(s)|^{2}+\mathbb{E}(W(x(s))))ds$$
and
$$D(t)=\int_{[\frac{t}{\tau}]\tau}^{t}e^{\varepsilon s}(\mathbb{E}|x(s)|^{2}+\mathbb{E}(W(x(s))))ds.$$

Then (\ref{44}) implies that for any $t\in[nT,nT+\theta),$
\begin{equation}\label{440}
\begin{aligned}
&\quad e^{\varepsilon t}\mathbb{E}V(\hat{x}_{t},\hat{r}_{t},t) + C_4\left(\sum_{i=n(m+m')+1}^{[\frac{t}{\tau}]}D_i +D(t)\right)\\&
 \le e^{\varepsilon nT}\mathbb{E}V(\hat{x}_{nT},\hat{r}_{nT},nT) +C_4D_{n(m+m')}\\&
 \le e^{\varepsilon nT}\mathbb{E}V(\hat{x}_{nT},\hat{r}_{nT},nT)+C_4\sum_{i=n(m+m')+1}^{[\frac{t}{\tau}]}D_i .
\end{aligned}
\end{equation}

And (\ref{43'''}) implies for any $t\in[nT+\theta,(n+1)T),$
\begin{equation}\label{440''}\aligned &\quad e^{\varepsilon t}\mathbb{E}V(\hat{x}_{t},\hat{r}_{t},t)+C_4\left(\sum_{i=n(m+m')+m+1}^{[\frac{t}{\tau}]}D_i +D(t)\right)\\&\le e^{\varepsilon (nT+\theta)}\mathbb{E}V(\hat{x}_{nT+\theta},\hat{r}_{nT+\theta},nT+\theta)\\&\quad+C_4D_{n(m+m')+m}+C_5
			\int_{nT+\theta}^te^{\varepsilon s}\mathbb{E}|x(s)|^{2}ds\endaligned\end{equation}
where by Condition \ref{c3} (\ref{c31}) and Condition \ref{c2}, we can choose $\varepsilon$ sufficiently small such that
$$\aligned&
C_1:=B_1-\gamma_{5}h^*e^{\varepsilon\tau}>0,\\&
C_2:=1-\gamma_{6}h^*e^{\varepsilon\tau}-\frac{\varepsilon a_{3}}{\gamma_7}>0,\\&
C_3:=1-\frac{\varepsilon a_3}{\gamma_7}-\gamma'_6h^*e^{\varepsilon\tau}>0,
\endaligned$$
$$C_4:=(\gamma_5\vee\gamma_6\vee\gamma'_{5}\vee\gamma'_6)h^*e^{\varepsilon\tau}\le C_1\wedge C_2\wedge C_3$$
and
$$C_5:=\gamma'_{4}+\varepsilon(a_2+a_3)+(\gamma'_{5}-\gamma'_6)h^*e^{\varepsilon\tau}+1-\frac{\varepsilon a_3}{\gamma_7}>0.$$

Then Gronwall Lemma implies that for any $t\in[nT+\theta,(n+1)T),$
\begin{equation}\label{441}
\begin{aligned}
 e^{\varepsilon t}\mathbb{E}V(\hat{x}_{t},\hat{r}_{t},t) +  C_4\left(\sum_{i=n(m+m')+m+1}^{[\frac{t}{\tau}]}D_i +D(t)\right)
&\le \left[e^{\varepsilon (nT+\theta)}\mathbb{E}V(\hat{x}_{nT+\theta},\hat{r}_{nT+\theta},nT+\theta)\right.\\&\quad \left.+ C_4D_{n(m+m')+m}\right] \cdot e^{\frac{C_5}{a_1}(t-(nT+\theta))}.
\end{aligned}
\end{equation}

Denote
$$b_{1,n}(t)=e^{\varepsilon t}\mathbb{E}V(\hat{x}_{t},\hat{r}_{t},t) +  C_4\left(\sum_{i=n(m+m')+1}^{[\frac{t}{\tau}]}D_i +D(t)\right),t\in[nT,nT+\theta)$$
$$b_{2,n}(t)=e^{\varepsilon t}\mathbb{E}V(\hat{x}_{t},\hat{r}_{t},t) +  C_4\left(\sum_{i=n(m+m')+m+1}^{[\frac{t}{\tau}]}D_i +D(t)\right),t\in[nT+\theta,(n+1)T)$$
 $$x_n=e^{\varepsilon nT}\mathbb{E}V(\hat{x}_{nT},\hat{r}_{nT},nT)+ C_4\sum_{i=(n-1)(m+m')+m+1}^{n(m+m')}D_i, $$ $$y_n=e^{\varepsilon (nT+\theta)}\mathbb{E}V(\hat{x}_{nT+\theta},\hat{r}_{nT+\theta},nT+\theta)+ C_4\sum_{i=n(m+m')+1}^{n(m+m')+m}D_i.$$

 Then (\ref{440}) and (\ref{441}) implies
$$b_{1,n}(t)\le x_n,\quad t\in[nT,nT+\theta),$$
$$b_{2,n}(t)\le y_n e^{\frac{C_5}{a_1}(t-(nT+\theta))},\quad t\in[nT+\theta,(n+1)T).$$

By the definition of $V$, it is obvious that
$$\aligned y_n&=b_{2,n}(nT+\theta)\\&
=e^{\varepsilon (nT+\theta)}\mathbb{E}V(\hat{x}_{nT+\theta},\hat{r}_{nT+\theta},nT+\theta)\\&
\le\lim_{t\uparrow nT+\theta}b_{1,n}(t) \\&
\le x_n.\endaligned$$

On the other hand, since control term $u$ does not appear in $[nT+\theta,(n+1)T)$, and (\ref{c21'}) and (\ref{pq}) imply that
$$\aligned e^{\varepsilon (n+1)T}\int^{(n+1)T}_{(n+1)T-\delta}\eta(s)ds&\le e^{\varepsilon\delta} \int^{(n+1)T}_{(n+1)T-\delta}e^{\varepsilon s}\eta(s)ds\\&
\le Ce^{\varepsilon\delta} \int^{(n+1)T}_{(n+1)T-\delta}e^{\varepsilon s}(\mathbb{E}|x(s)|^{2}+\mathbb{E}(W(x(s))))ds\\&
\quad+Ce^{\varepsilon\delta} \int^{(n+1)T}_{(n+1)T-\delta}e^{\varepsilon s}(\mathbb{E}|x_h(s)|^{2}+\mathbb{E}(W(x_h(s))))ds\\&
\le Ce^{\varepsilon\delta}\sum_{i=n(m+m')+m+1}^{(n+1)(m+m')}D_i,\endaligned$$

then we have
$$\aligned &\quad x_{n+1}\\&=e^{\varepsilon (n+1)T}\mathbb{E}V(\hat{x}_{(n+1)T},\hat{r}_{(n+1)T},(n+1)T)+ C_4\sum_{i=n(m+m')+m+1}^{(n+1)(m+m')}D_i\\&
\le\lim_{t\uparrow(n+1)T} e^{\varepsilon t}\left(\mathbb{E}V(\hat{x}_{t},\hat{r}_{t},t)+\frac{2L^{2}\delta e^{\varepsilon\delta}}{\gamma_{1}}\int^{t}_{t-\delta}\eta(s)ds\right)+ C_4\sum_{i=n(m+m')+m+1}^{(n+1)(m+m')}D_i\\&
\le \lim_{t\uparrow(n+1)T} e^{\varepsilon t}\mathbb{E}V(\hat{x}_{t},\hat{r}_{t},t)+(\frac{2CL^{2}\delta e^{\varepsilon\delta}}{\gamma_{1}}+ C_4)\sum_{i=n(m+m')+m+1}^{(n+1)(m+m')}D_i\\&
\le (1+\frac{2CL^{2}\delta e^{\varepsilon\delta}}{C_4\gamma_{1}})\left[e^{\varepsilon (nT+\theta)}\mathbb{E}V(\hat{x}_{nT+\theta},\hat{r}_{nT+\theta},nT+\theta)+ C_4D_{n(m+m')+m}\right] \cdot e^{\frac{C_5}{a_1}(t-(nT+\theta))}\\&
\le(1+\frac{2CL^{2}\delta e^{\varepsilon\delta}}{C_4\gamma_{1}})e^{\frac{C_5}{a_1}(T-\theta)}y_n.\endaligned$$

Therefore,
$$x_{n+1}\le (1+\frac{2CL^{2}\delta e^{\varepsilon\delta}}{C_4\gamma_{1}})e^{\frac{C_5}{a_1}(T-\theta)}x_n.$$

Denote $A=(1+\frac{2CL^{2}\delta e^{\varepsilon\delta}}{C_4\gamma_{1}})e^{\frac{C_5}{a_1}(T-\theta)}.$ Then we have
$$x_n\le A^n x_0.$$

So for any $t\in [nT, (n+1)T), n\ge 0,$ we have
$$b_{1,n}(t)\vee b_{2,n}(t)\le  A^ne^{\frac{C_5}{a_1}(T-\theta)} x_0.$$

It follows that
$$\aligned a_1\mathbb{E}|x(t)|^2&\le\mathbb{E}V(\hat{x}_{t},\hat{r}_{t},t)\\&
\le   A^ne^{\frac{C_5}{a_1}(T-\theta)} x_0e^{-\varepsilon t}\\&
=e^{\frac{C_5}{a_1}(T-\theta)} x_0 e^{-\mu t}e^{n(\frac{C_5}{a_1}(T-\theta)+\log (1+\frac{2CL^{2}\delta e^{\varepsilon\delta}}{C_4\gamma_{1}}))-(\varepsilon-\mu) t}\\&
\le e^{\frac{C_5}{a_1}(T-\theta)} x_0 e^{-\mu t}e^{n(\frac{C_5}{a_1}(T-\theta)+\log (1+\frac{2CL^{2}\delta e^{\varepsilon\delta}}{C_4\gamma_{1}}))-n(\varepsilon-\mu)T}.\endaligned$$

Then for $\theta>(1-\frac{a_1\varepsilon}{C_5})T+\frac{a_1\log (1+\frac{2CL^{2}\delta e^{\varepsilon\delta}}{C_4\gamma_{1}})}{C_5},$ we can choose $$\mu=\varepsilon-\frac{C_5}{a_1}(1-\frac{\theta}{T})-\frac{\log (1+\frac{2CL^{2}\delta e^{\varepsilon\delta}}{C_4\gamma_{1}})}{T}$$ (obviously, $0<\mu<\varepsilon$) such that
\begin{equation}\label{70}\mathbb{E}|x(t)|^2\le \frac{e^{\frac{C_5}{a_1}(T-\theta)} x_0}{a_1}e^{-\mu t}.\end{equation}

Moreover, by H\"{o}lder inequality, Theorem \ref{youjie1} and \eqref{70}, we get
	\begin{equation}\label{71}
		\begin{aligned}
			\mathbb{E}|x(t)|^{\bar{q}}&\leq  (\mathbb{E}|x(t)|^{2})^{\frac{q-\bar{q}}{q-2}}
			\cdot(\mathbb{E}|x(t)|^{q})^{\frac{\bar{q}-2}{q-2}}\\
            &\leq
			\left(\frac{e^{\frac{C_5}{a_1}(T-\theta)} x_0}{a_1}e^{-\mu t}\right)^{\frac{q-\bar{q}}{q-2}}\cdot C'^{\frac{\bar{q}-2}{q-2}} \\
			&\leq C'^{\frac{\bar{q}-2}{q-2}}\cdot \left(\frac{e^{\frac{C_5}{a_1}(T-\theta)} x_0}{a_1}\right)^{\frac{q-\bar{q}}{q-2}}\cdot e^{-\frac{q-\bar{q}}{q-2}\mu t}.\\
		\end{aligned}
	\end{equation}
where $C':=\sup_{t\ge0}\mathbb{E}|x(t)|^q<\infty.$

The proof is complete.  \hfill$\Box$

\section{An example}

Now let us present an example to illustrate our theory.
	
\textbf{Example} Consider the follow equation
\begin{equation}\label{sde1}
\begin{aligned}
dx(t) &= f(x(t),x(t-h(t)),r(t),t)dt  \\
&\quad +g(x(t),x(t-h(t)),r(t),t)dB(t)
\end{aligned}
\end{equation}
where $r(t)$ is a Markov chain with state space $S=\{1,2\}$ and generator $\Gamma = \begin{pmatrix} -2 & 2 \\ 1 & -1 \end{pmatrix}$.
\[
f(x, y, 1, t) = 0.5x - 12x^3 + 0.2y + 0.5y^3,\quad
f(x, y, 2, t) = 0.8x - 15x^3 + 0.4y + 0.8y^3
\]
\[
g(x, y, 1, t) = 0.4y + 0.5y^2, \quad
g(x, y, 2, t) = 0.5y + 0.6y^2
\]
\[
h(t) = 0.15 + 0.05 \sum_{k=0}^\infty (-1)^k (t-k)I_{[k,k+1)}(t).
\]
That is, (\ref{duoxiangshi}) holds for $K=15.$

It is easy to verify that
\begin{equation}\label{your_label}
\begin{aligned}
&\quad xf(x,y,i,t) + 3|g(x,y,i,t)|^2 \le |x|^2 + 1.85|y|^2 - 11.875|x|^4 + 2.58|y|^4,
\end{aligned}
\end{equation}
i.e. Assumption \ref{Khas} holds for $p=4, q=7, q_1=q_2=3, q_3=q_4=2$ with $K'=1.85, \alpha_1=11.875$ and $\alpha_2=2.58.$ Obviously, this condition can not guarantee the stability of $x(t)$.

Now let	us consider the control function
\[
u(x, 1, t) = -8x,\quad u(x, 2, t) = -9x
\]	
($L=9$ in (\ref{uu})) and the controlled system
\begin{equation}\label{jsde1}
\begin{aligned}
dx(t) &= \left(f(x(t),x(t-h(t)),r(t),t) + u(x(v(t)),r(v(t)),t)I(t)\right)dt  \\
&\quad +g(x(t),x(t-h(t)),r(t),t)dB(t)
\end{aligned}
\end{equation}
where $I(t)=\sum_{k=1}^\infty I_{[k,k+\theta)}(t), \theta\in(0,1).$

Then it follows that Condition 4.1 holds for
$$k_{11} = -7.4, \, l_{11} = 0.26, \, \beta_{11} = 11.875, \, \alpha_{11} = 0.625,$$
$$k_{21} = -7.4, \, l_{21} = 0.58, \, \beta_{21} = 11.875, \, \alpha_{21} = 1.125,$$
$$k_{12} = -8, \, l_{12} = 0.45, \, \beta_{12} = 14.8, \, \alpha_{12} = 0.96,$$
$$k_{22} = -8, \, l_{22} = 0.95, \, \beta_{22} = 14.8, \, \alpha_{22} = 1.68.$$

We can also compute
\[\aligned &
\left( \theta_1, \theta_2 \right)^T = (0.067, 0.063)^T,\\&
\left( \bar{\theta}_1, \bar{\theta}_2 \right)^T = (0.0336, 0.0313)^T.\endaligned
\]

Now let $U, LU, L'U$ be defined as in (\ref{test}), (\ref{LU}) and  (\ref{L'U}). Then for \(i=1\), we can choose $\gamma_i>0, i=1,2,3$ sufficiently small such that all the coefficients of the crossing terms are positive, so
\begin{align*}
& LU(x,y,1,t) + \gamma_1 (2\theta_1|x| + 4\bar{\theta}_1|x|^3)^2 + \gamma_2 |f(x,y,1,t)|^2 + \gamma_3 |g(x,y,1,t)|^2 \\
=&\ (-1.013 + 0.017956\gamma_1 + 0.25\gamma_2)x^2 + (-2.6206 + 0.0360192\gamma_1 - 12\gamma_2)x^4 \\
& + (-1.6128 + 0.01806336\gamma_1 + 144\gamma_2)x^6  + (0.0268 + 0.2\gamma_2)xy  + (0.067 + 0.5\gamma_2)x y^3 \\
& + (0.02688 - 4.8\gamma_2)x^3 y  + (0.0672 - 12\gamma_2)x^3 y^3 \\
& + 0.032256 x^2 y^2  + 0.08064 x^2 y^3 + 0.0504 x^2 y^4 + (0.01072 + 0.04\gamma_2 + 0.16\gamma_3)y^2 \\& + (0.0268 + 0.4\gamma_3)y^3
+ (0.01675 + 0.2\gamma_2 + 0.25\gamma_3)y^4  + 0.25\gamma_2 y^6 \\
\leq&\ \left[-1.013 + 0.017956\gamma_1 + 0.25\gamma_2+ \frac{0.0268 + 0.2\gamma_2}{2}\right]x^2 \\
& + \left[-2.6206 + 0.0360192\gamma_1 - 12\gamma_2 + \frac{0.067 + 0.5\gamma_2}{4}  \right. \\
& \quad\left. +\frac{3(0.02688 - 4.8\gamma_2)}{4} + \frac{0.032256}{2} + 0.02016\right]x^4 \\
& + \left[-1.6128 + 0.01806336\gamma_1 + 144\gamma_2 + \frac{0.0672 - 12\gamma_2}{2} +\frac{0.08064}{6} + 0.0168\right]x^6\\&
 + \left[0.01072 + 0.04\gamma_2 + 0.16\gamma_3  +\frac{0.0268 + 0.2\gamma_2}{2}\quad+\frac{0.0268 + 0.4\gamma_3}{2}\right]y^2  \\
& + \left[0.01675 + 0.2\gamma_2 +0.25\gamma_3 + \frac{3(0.067 + 0.5\gamma_2) + 0.02688 - 4.8\gamma_2}{4}\right. \\
&\quad\left.+ \frac{0.032256}{2} + 0.02016 + \frac{0.0268 + 0.4\gamma_3}{2}\right]y^4 \\
& + \left(0.25\gamma_2 + \frac{0.0672 - 12\gamma_2}{2} + \frac{0.08064}{3} + 0.0336\right) y^6 \\
=&\ -(0.9996 - 0.017956\gamma_1 - 0.35\gamma_2)x^2 + (0.03752 + 0.14\gamma_2 + 0.36\gamma_3) y^2 \\
& - (2.547402 - 0.0360192\gamma_1 + 15.475\gamma_2) x^4 + (0.123408 - 0.625\gamma_2 + 0.45\gamma_3) y^4 \\
& - (1.54896 - 0.01806336\gamma_1 - 138\gamma_2) x^6  + (0.09408 - 5.75\gamma_2) y^6
\end{align*}
where we have used $$x^2y^3\le \frac{1}{2}x^2(y^2+y^4)\le\frac{x^4+y^4}{4}+\frac{x^6+2y^6}{6}$$
in the first inequality.
	
Similarly, for \(i=2\),
\begin{align*}
&\quad  LU(x,y,2,t) + \gamma_1 (2\theta_2|x| + 4\bar{\theta}_2|x|^3)^2+ \gamma_2 |f(x,y,2,t)|^2 + \gamma_3 |g(x,y,2,t)|^2 \\
&\le \left(-1.004 + 0.015876\gamma_1 + 0.96\gamma_2\right)x^2 + \left(0.05985 + 0.48\gamma_2 + 0.55\gamma_3\right)y^2 \\
& \quad+ \left(-2.799935 + 0.0315504\gamma_1 - 32.68\gamma_2\right)x^4 + \left(0.181345 - 1.4\gamma_2 + 0.66\gamma_3\right)y^4 \\
& \quad+ \left(-1.786604 + 0.01567504\gamma_1 + 213\gamma_2\right)x^6 + \left(0.113932 - 11.36\gamma_2\right)y^6
\end{align*}
and
\begin{equation*}
\begin{aligned}
L'U(x,y,1,t)
&\le 0.0724x^2 - 1.472202x^4 - 1.54896x^6 \\
& \quad + 0.05768y^2 + 0.103248y^4 + 0.09408y^6.
\end{aligned}
\end{equation*}
\begin{equation*}
\begin{aligned}
 L'U(x,y,2,t)
&\le 0.13x^2 - 1.673135x^4 - 1.786604x^6 \\
&\quad  + 0.05985y^2 + 0.181345y^4 + 0.132712y^6.
\end{aligned}
\end{equation*}

Let $\gamma_1=1, \gamma_2=0.001, \gamma_3=0.002$. Then $$-1.4\gamma_2 + 0.66\gamma_3\le 0,$$
\begin{equation*}
\begin{aligned}
&\quad (2.547402 - 0.0360192\gamma_1 + 15.475\gamma_2) \wedge(2.799935 - 0.0315504\gamma_1 + 32.68\gamma_2) \\
&= 2.5268578 \ge 1.472202
\end{aligned}
\end{equation*}
and
$$\frac{0.181345}{0.132712}\bigwedge\frac{0.181345-1.4\gamma_2 + 0.66\gamma_3}{0.113932 -11.36\gamma_2}=1.366489,$$
$$\frac{1.472202}{1.54896-0.01806336\gamma_1-138\gamma_2}=1.056936.$$

Then let
\begin{equation*}
\begin{aligned}
W(x) &= 1.472202|x|^4  +(1.54896 - 0.01806336\gamma_1 - 138\gamma_2)|x|^6 \\
&= 1.472202|x|^4 + 1.39289664|x|^6.
\end{aligned}
\end{equation*}

We can check that Condition \ref{c2} holds for
$$\gamma_4=0.9996-0.017956\gamma_1-0.35\gamma_2=0.981294,$$
$$\gamma_5=0.05985 + 0.48\gamma_2 + 0.55\gamma_3=0.06143,$$
$$\gamma_6=\frac{0.181345-1.4\gamma_2 + 0.66\gamma_3}{1.472202}\approx0.123112,$$
$$ \gamma_4'=0.13,\ \gamma_5'=0.05985,\ \gamma'_6=\frac{0.181345}{1.472202}\approx0.123160,$$
$$\gamma_7=1.54896-0.01806336\gamma_1-138\gamma_2=1.39289664$$
and
$$\gamma_8=3.021162-0.01806336\gamma_1-138\gamma_2=2.86509864.$$

Moreover, $$1\wedge\gamma_4=0.981294>0.49264\approx2(\gamma_5\vee\gamma_6\vee\gamma'_5\vee\gamma'_6)h^*$$ holds for such $\gamma_i, i=1,2,3$.
Since $h^*=\frac{20}{19}+\frac{20}{21}=\frac{800}{399}$ and $\tau=0.2$, then by Theorem \ref{dl1}, we can choose $\delta=10^{-5}$ such that Condition \ref{c3} holds. We can verify that if we choose $\varepsilon=3.13428,$ then $C_1\wedge C_2\wedge C_3\approx0.462197114>0.462196735=C_4$ and $C_5\approx1.1321$. Then for any $\theta\in(1-\frac{a_1\varepsilon}{C_5}+\frac{a_1\log (1+\frac{2CL^{2}\delta}{C_4\gamma_{1}})}{C_5},1]\approx(0.8089,1]$, the controlled system (\ref{jsde1}) is exponentially stable in $L^q$ for any $q<7$. Moreover, if we choose $\theta=0.9,$ the mean square exponential stability convergence rate can be chosen as $\mu=\varepsilon-\frac{C_5}{a_1}(1-\frac{\theta}{T})-\frac{\log (1+\frac{2CL^{2}\delta}{C_4\gamma_{1}})}{T}\approx 1.0368$ and $L^q$ exponential stability convergence rate is $\frac{7-q}{5}\mu$. If we choose $\theta$ larger, e.g. $\theta=0.95,$ then the mean square exponential stability convergence rate can reach $1.9353$. That is, if $\theta$ is small (for example $\theta<0.8$), Theorem \ref{dl1} can not be applied to obtain the exponential stability of the controlled system, if $\theta>0.8089$ and becomes larger and larger, then the convergence rate is also larger, which indicates the proportion of control interval has positive relation to the convergence rate.

If we choose $x_0=\xi(s)=0.5, s\in[-0.2,0], r(0)=1, \delta=10^{-5}$ and stepsize $\Delta=0.001$, then computer simulation of the uncontrolled system (\ref{sde1}) and controlled system (\ref{jsde1}) are shown in Fig. 1 and 2, which support our conclusions.

 \begin{figure}[H]
 \centering
  \includegraphics[width=0.5\linewidth]{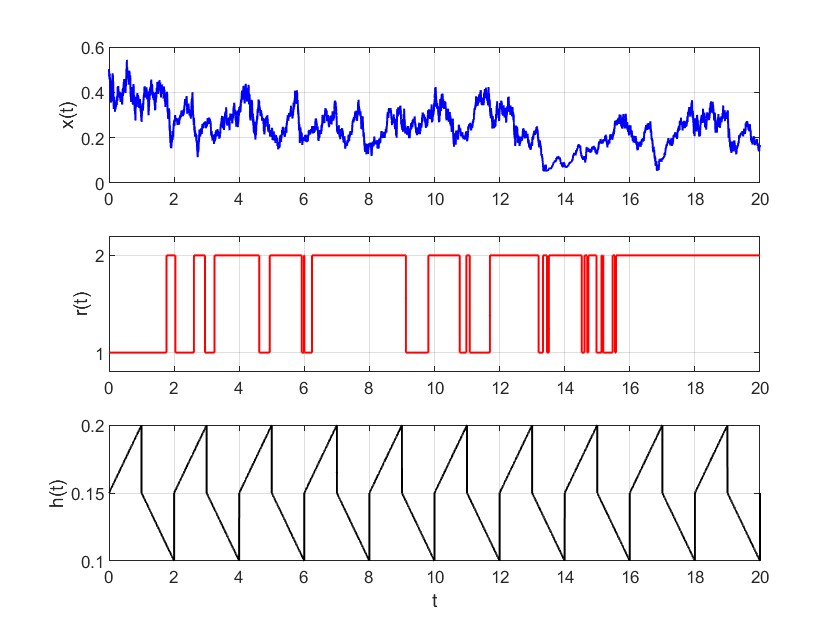}
  \caption{The sample path of uncontrolled hybrid stochastic system (\ref{sde1})}
 \label{tu1}
 \end{figure}

 \begin{figure}[H]
 \centering
  \includegraphics[width=0.5\linewidth]{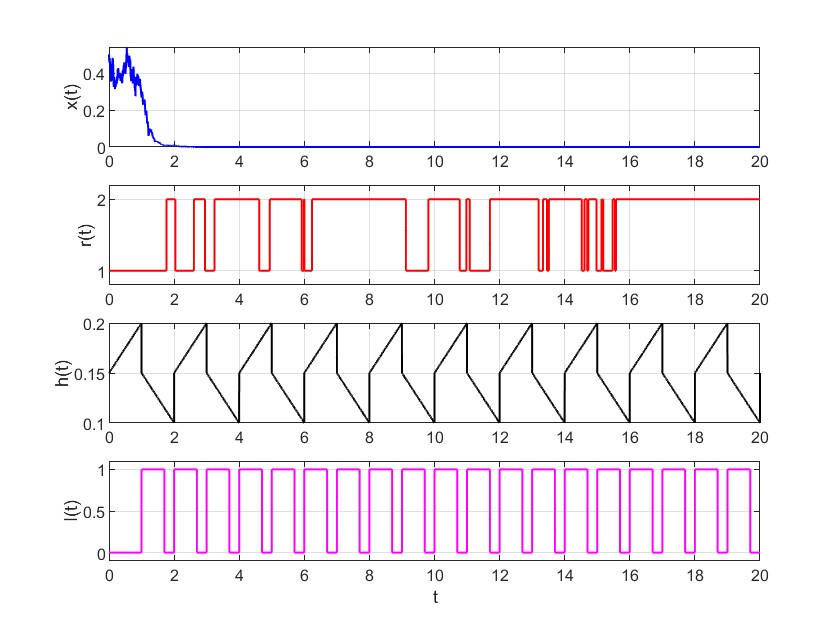}
  \caption{The sample path of controlled hybrid stochastic system (\ref{jsde1})}
 \label{tu2}
 \end{figure}

\section{Concluding Remark}

In this paper, we stabilize the highly nonlinear hybrid stochastic differential delay equations by using a periodically intermittent controller based on discrete time state observations with asynchronous switching. Both the drift term and diffusion terms of the equations only satisfy polynomial growth conditions. The stabilization conditions of control function and coefficients of the system are presented. Moreover, we obtain the moment exponential stability of the controlled system by introducing new proof method to overcome the effect of asynchronous switching and intermittent of the controller. Finally, the upper bound of discrete observation period as well as the lower bound of the control width are obtained.

\end{document}